\documentclass{article}
\usepackage{amsmath,amssymb,latexsym,amsthm,url}
\usepackage{graphicx}
\usepackage{enumerate}

\newtheorem{theorem}{Theorem}[section]

\newtheorem{lemma}[theorem]{Lemma}

\theoremstyle{definition}

\newcommand\calS{{\mathcal S}}
\newcommand\calM{{\mathcal M}}
\newcommand\calA{{\mathcal A}}
\newcommand\calB{{\mathcal B}}
\newcommand\calC{{\mathcal C}}
\newcommand\calV{{\mathcal V}}

\newcommand\subdir{\setbox0=\hbox{$\mathchar "0362$}%
\mathop{\displaystyle\smash{\raise 0.3\ht0\hbox{\hbox to \wd0%
{\hfil\vrule height4pt width0.4pt\hfil}\kern-\wd0%
\lower\ht0\box0}}}
}
\newcommand\dogpind[1]{\left[\vphantom{\hbox{${#1}'$}}#1\right]}
\newcommand\gpind[2]{\dogpind{#1{:}#2}}

\newcommand\gen[1]{\langle#1\rangle}
\newcommand\sz[1]{\left|#1\right|}

\newcommand{\dnk}[3]{\lower 0.2ex\hbox{$#1$}\kern -.1em{\setminus}
                     \kern -.1em\raise 0.2ex\hbox{$#2$}
                     \kern -.1em{/}\kern -.1em\lower 0.2ex\hbox{$#3$}}

\title{Finding Intermediate Subgroups}
\author{Alexander Hulpke\\
Department of Mathematics\\
Colorado State University\\
1874 Campus Delivery\\
Fort~Collins, CO, 80523-1874, USA\\
\url{hulpke@colostate.edu}}

\begin{document}

\maketitle

\begin{abstract}
This article describes a practical approach for determining the lattice of
subgroups $U<V<G$ between given subgroups $U$ and $G$, provided the total
number of such subgroups is not too large.
It builds on existing functionality for element conjugacy, double cosets and
maximal subgroups.
\end{abstract}

%

\section{Introduction}

The question of determining the subgroups of a given finite group (possibly
up to conjugacy) has been
of interest since the earliest days of computational group
theory~\cite{neubueser60}
with current algorithms~\cite{hulpkeinvar,coxcannonholtlatt,hulpketfsub} being able to
work easily for groups of order several million. A fundamental difficulty
however is presented if the group has 
a large number of subgroups, overwhelming available storage.
\smallskip

This limitation indicates the need for more specific algorithms that
determine only part of the subgroup lattice of a group.
Beyond specific subgroups (such as Sylow- or Hall-), the main result in this
direction is that of maximal
subgroups~\cite{eickhulpke01,holtcannonmaxgroup}. Such a calculation can be
iterated to yield subgroups of small index~\cite{cannonholtslatterysteel}.

(The case of minimal subgroups is simply that of conjugacy classes of
elements of prime order and thus is solved.)
\medskip

In this paper we will consider the related question of {\em intermediate
subgroups}, that is given a group $G$ and a subgroup $U\le G$, we aim to
determine all subgroups $U\le V\le G$ (or, in some cases, one if it exists), as
well as inclusions amongst these subgroups.

In some situations, for example if $U\lhd G$,
or if $U$ is cyclic of small prime order, this amounts to determining a
significant part of a subgroup lattice of (a factor of) $G$ and most likely
the best approach will be simply to calculate all subgroups and then to
filter for the
desired candidates. We thus shall implicitly assume that we are in a
situation in which the number of intermediate subgroups is small.
\medskip

Such functionality has numerous applications, in some cases requiring all
intermediate subgroups, in others only one. To name just a few of them:
\begin{itemize}
\item
Birkhoff's question on the possibility of representing finite lattices as intervals in congruence
lattices has been of long standing interest in universal algebra, starting
with~\cite{graetzerschmidt63}.
Subsequently, \cite{palfypudlak} reduced the problem to
that of realizing a finite lattice as interval in a finite subgroup lattice,
prompting significant investigation of this question
interest~\cite{palfy95,watatani96,shareshian03,aschbacher08}. With no
positive answer for even small lattices ---
figure~\ref{latpic} depicts the smallest lattice that is not know to occur
as interval in a finite subgroup lattice ---
researchers have tried to gain insight by trying experimentally to to
realize particular cases on the computer~\cite{deMeoThesis}.
This requires an algorithm for intermediate subgroups.
\begin{figure}
\begin{center}
\includegraphics[width=4cm]{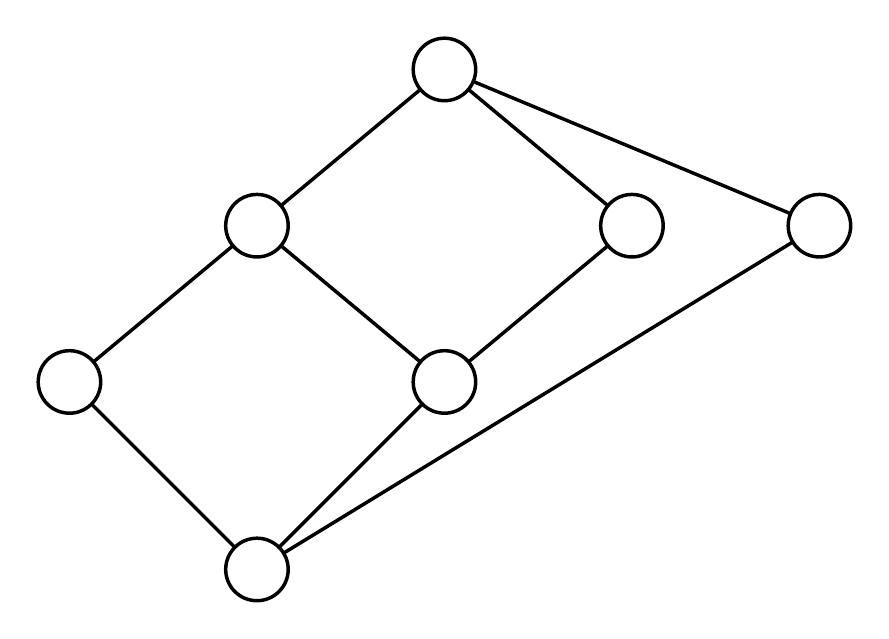}
\end{center}
\caption{The smallest open case for the lattice representation problem}
\label{latpic}
\end{figure}
\item
When representing a transversal of right cosets in permutation groups,
the standard approach of reduction to stabilizers~\cite{dixonmajeed}
can, but does not have to, reduce the storage
requirements. An example of such a situation is the transversal of a
(cyclic) 11-Sylow subgroup $U\le S_{11}$. This subgroup is transitive and
any proper subgroup trivial, thus transition to a stabilizer would require
handling every group element at some point.

Instead, utilizing the intermediate subgroup $M_{11}$ reduces storage
requirements from the index $3628800$ of $U$ down to $5040+720=5760$ coset
representatives.
\item
Double cosets, in particular in permutation groups, are a basic tool of
combinatorial enumeration as they describe how a group orbit splits up under
reduction to a subgroup. The standard approach~\cite{laue82,schmalz90}
utilizes a chain of intermediate subgroups, reducing iteratively the
calculation of $\dnk{A}{G}{B}$ to $\dnk{A_1}{G}{B}$ for $A<A_1$ and action
of $B$ on cosets of $A$.
Knowledge of intermediate subgroups roughly logarithmizes the cost of
calculations.
\item
For infinite matrix groups, proving a subgroup to be of finite index of a
subgroup (such as in~\cite{longreid11}) often reduces to a coset
enumeration.  For huge indices this is infeasible for memory reasons.  If
the expected index has been calculated~\cite{detinkoflanneryhulpke17}
through a subgroup $U<Q$ in a suitable finite quotient $Q$ of the group, an
intermediate subgroup $U<V<Q$ can be used to first rewrite the
presentation to (the pre-image of) $V$, thus opening the possibility for a
smaller index coset enumeration.

\item
The following observation however indicates that the determination of
intermediate subgroups should be expected to be hard in general:
When computing stabilizers
under group actions, often a small number of Schreier generators (which --
by the birthday paradox -- should arise after enumerating roughly the square
root of the orbit length) will generate the full stabilizer, but this is
only proven after enumerating a significant part (typically $1/p$, where $p$
is the smallest prime divisor of the stabilizer index) of the orbit,
thus establishing that the stabilizer indeed cannot be larger.

Knowledge of intermediate subgroups would allow verification of this fact,
thus establishing the stabilizer with less effort by showing that no
subgroup above the presumptive stabilizer in fact stabilizes.
\end{itemize}

Now assume that $G\ge U$ are given, and we want to enumerate the subgroups
$G>V>U$. A basic approach follows from the observation that
these subgroups correspond to block systems for the action of $G$
on the cosets of $U$. If $\gpind{G}{U}$ is small (in practice not more than a few
hundred), it is possible to determine the permutation representation
on the cosets explicitly. One then can utilize a block-finding algorithm~\cite{bealsseress92} as
a tool for finding all intermediate subgroups. This however becomes
impractical if the index gets larger, as the examples below indicate.

\section{A maximal subgroups based approach}

Instead, we shall rely on existing algorithms for maximal
subgroups~\cite{holtcannonmaxgroup,eickhulpke01} to find (conjugates of) subgroups
lying below $G$ and above $U$. 

(An alternative dual approach would be to utilize minimal supergroups, but
an algorithm for these does not yet exist.)

A second tool will be routines to determine element centralizers and test for
element conjugacy. 

Finally, the approach itself involves the calculation of double cosets. When
using this routine as a tool as part of a double coset computation (as
suggested in the introduction) there thus is a priori the potential of an
infinite recursion. This will be in general avoided by the fact that the
required double coset calculations involve strictly smaller subgroups;
however in a general purpose routine this needs to be checked for.
\bigskip

\medskip

This approach produces a new algorithm ${\tt IntermediateSubgroups}(G,U)$,
that takes as input two groups $G\ge U$ and returns as output a list of the intermediate
subgroups $G>V>U$, as well as the maximality inclusions amongst all subgroups
$G\ge V\ge U$.

The main tool for this calculation will be a further new algorithm:

${\tt
EmbeddingConjugates}(G,A,B)$ (which we shall describe later) takes as
arguments a
group $G$ and $B,A\le G$ and returns a list of
the $G$-conjugates of $A$ containing $B$, that is the subgroups $A^g>B$
for $g\in G$, together with the respective conjugating element $g$.

\subsection{Intermediate Subgroups Algorithm}

The algorithm takes as input the groups $G\ge U$ and returns a list 
$\calS$ of intermediate subgroups (including $G$), as well as a list $I$ of maximality
inclusion relations amongst these subgroups,

\begin{enumerate}[1.]
\item
Initialize $\calS:=[G]$; let $I=[\,]$.
\item
\label{step2}
While there is a subgroup $T\in\calS$ that has not been processed,
perform the following steps for $T$, otherwise end and return $\calS$ and $I$.
\item 
If $\gpind{T}{U}$ is prime, record $U<T$ as a maximality relation in $I$, mark $T$ as
processed, and go back to
step~\ref{step2}.
\item 
\label{step4}
Determine a list $\calM$ of $T$-representatives of the maximal subgroups of $T$
whose order is a multiple of $\sz{U}$.
\item
\label{step5}
For each subgroup $W\in \calM$, let $\calV_W={\tt
EmbeddingConjugates}(T,W,U)$ (that is the $T$-conjugates of $W$ containing $U$).
Add every subgroup $X\in \calV_W$ to $\calS$ (unless it is already in the
list), and add the maximality relations
$X<T$ to $I$.
\item If all sets $\calV_W$ were empty, add the maximality relation $U<T$ to
$I$.
\item 
Mark $T$ as processed, and go back to
step~\ref{step2}.
\end{enumerate}
If only one intermediate subgroup (or a maximality test) is required, the
calculation can stop in step~\ref{step5} once a single subgroup $X\in\calV$
has been found.

\begin{proof}
To see the correctness of this algorithm we notice that the only subgroups
added to $\calS$ are intermediate, and that a maximality relation is only
recorded if a subgroup is contained maximally in another. The algorithm thus
returns a list of intermediate subgroups and valid maximality relations.

To show completeness of the lists returned, consider an intermediate
subgroup $U< V< G$. If $V$ is maximal in $G$ it will be found as maximal
subgroup of $T=G$ in step~\ref{step5}.

Otherwise there will be subgroups $V<W<G$ in which $V$ is contained
maximally. By induction over the index in $G$ we may assume that these
subgroups are included in $\calS$. Then $V$ is obtained as a subgroup in
$\calV_W$ in step~\ref{step5}, and for each such subgroup $W$ the maximality
inclusion is recorded in $I$.

The same argument shows that any maximality relation amongst the subgroups
with be recorded in $I$.
\end{proof}

We note a few places of possible improvements:

In step~\ref{step4}, we
note that if $G$ is a permutation group, we can furthermore restrict $\calM$ to
those subgroups whose
orbit lengths on the permutation domain can be partitioned by those of $U$.
This gives a significant speedup in case of ``generic'' larger groups
such as $S_n$.
\smallskip

In the course of the calculation
several subgroups in $\calS$ may have been obtained as
conjugates of the same group. In this situation the maximal subgroup
calculation in step~\ref{step4} can transfer the list $\calM$ of
(representatives of) maximal subgroup representative from one subgroup
to another (as appropriate conjugates).

The cost of this algorithm is roughly proportional to the number of
intermediate subgroups, as every subgroup gets processed in the same way.
While it would be possible to carry partial information about maximal
subgroups through the iteration, this will not change the asymptotic
behavior.

The implicit assumption of few intermediate subgroups thus makes this
approach feasible. As the examples in section~\ref{xampl} show,
the method is feasible also in practice.

\subsection{Embedded Conjugates:}

To describe the the required subroutine, we consider 
first the following, related, problem:
Given a group $A$ and two subgroups $A,B\le G$, we are seeking to find all
conjugates $B^g\le A$ (that is representatives up to $A$-conjugacy thereof),
and for each conjugate subgroup $B^g$ a conjugating elements $g$.

If the $G$-orbit of $B$, that is $\gpind{G}{N_G(B)}$ is small,
we can simply determine this orbit up to
$A$-conjugacy, parameterized by double cosets $\dnk{N_G(B)}{G}{A}$, and test
which conjugates of $B$ lie in $A$.
\smallskip

If $G$ is a permutation group, it also might be possible to use a backtrack
search, similar to that of a normalizer
calculation~\cite{leon97,Theissen97}, to find elements that conjugate $B$
into $A$. Since the groups are of different order, however, existing
refinements would not be available, and the search therefore could easily
degenerate into testing all double cosets as just described. We thus have
not investigated such an approach further.
\medskip

For all other cases, we use an approach that is
motivated by the generic isomorphism search routine~\cite[\S9.3.1]{holtbook},\cite[V.5]{hulpkediss}, reducing
conjugation of subgroups to conjugation of elements:

Suppose that $B=\gen{b_1,b_2,\ldots,b_k}$. Any element $g$ conjugating $B$
into $A$ must (this is necessary and sufficient) map all of the $b_i$ to elements of $A$ and these images
must lie in conjugacy classes of $A$. 
\medskip

This yields the following algorithm, whose input is a $k$-element generating
sequence of $B$, together with the groups $G$ and $A$.
\smallskip

At the start we precompute for each index $i$ a list $\calA_i$ of those
$A$-conjugacy classes $\calC\subset A\cap b_i^G$.
We can do this based on a list of conjugacy classes
of $A$ and explicit element conjugacy tests.

If $A\cap b_i^G$ is empty for any $i$, we know that $B$ cannot be conjugated
into $A$.

In the case of a large $A$ and small $B$ it is often possible to select
generators $\{b_i\}$ with particular properties, say of prime order, which
allows to limit the conjugacy classes of $A$ which are required, in the
example of prime order it would be the classes intersecting a Sylow
subgroup.
\medskip

We now describe the main part of the algorithm, which is a recursive depth
first routine
${\tt Search}(C,i,g)$ that takes as parameters an index $i$,
a conjugating element $g\in G$ such that $b_j^g\in A$ for all $j<i$,
and a subgroup $C=C_G(b_1^g,b_2^g,\cdots,b_{i-1}^g)\le G$.

This routine is called once as ${\tt Search}(G,1,1_G)$.
It collects the result pairs $(B^g,g)$ in a global list
$\calB$ that is initialized to an empty list at the start.
\smallskip

The ${\tt Search}(C,i,g)$ routine then proceeds as follows:
\begin{enumerate}[1.]
\item
\label{schritt1}
If $i>k$ (that is all generators are mapped already) store $(B^g,g)$
in $\calB$, and return.
\item 
\label{schritt2}
Otherwise, let $x=b_i^g$ and set $D=C\cap A$.
\item
\label{schritt3}
For every class $Y\in\calA_i$, let $\{y_j\}_j$ be a set of
representatives of the $D$-classes that partition $Y$. These
representatives can be obtained by conjugating a fixed representative $y\in
Y$ with representatives of the double cosets $\dnk{C_A(y)}{A}{D}$.
\item For every representative $y_j$, test whether there is 
$d\in D$ such that $x^d=y_j$. If so, call recursively ${\tt
Search}(C_C(y_j),i+1,g\cdot d)$.
\item Iterate until all $y_j$ for all classes $Y\in\calA_i$ have been
tested. Afterwards return.
\end{enumerate}

If $C$ (and thus $D$) is small in step~\ref{schritt3}, the number of double
cosets could be large. In this case it could be worth to instead consider
$C$-conjugates of $x$, up to $D$-conjugacy (parameterized by the double
cosets $\dnk{C_C(x)}{C}{D}$) and test which conjugates lie in $A$.

\begin{lemma}
When the call to ${\tt Search}(G,1,1_G)$ returns, the list $\calB$ contains
pairs $(B^g,g)$ such that $B^g\le A$ and every conjugate $B^h\le A$ will be
$A$-conjugate to a subgroup $B^g$ that arises in $\calB$.
\end{lemma}
\begin{proof}
We first observe that whenever ${\tt
Search}(C,i,g)$ is called, we have that $b_j^g\in A$ for every $j<i$. Thus
the only results $(B^g,g)$ stored in step~\ref{schritt1} satisfy that
$b_j^g\in A$ for every $j$, that is $B^g\le A$ as required.
\smallskip

We also note that at stage $i$ of the calculation a recursive call for level
$i+1$ is done with a first argument centralizing $b_i^g$.

Thus at every stage of the calculation we have that $C\le C_G(b_j^g)$ for
$j<i$, and therefore the conjugating element $g$ is only modified by a
factor in $C_G(b_1^g,\ldots,b_{i-i}^g)$.
Thus the images $b_j^g$ for $j<i$ remain fixed in stage $i$.
\smallskip

Now assume that $B^h\le A$ for some element $h\in G$. Then $b_i^h\in A$ for
all $i$. Thus there exists $c_1\in A$ such that $b_1^{hc_1}$ is the chosen
representative of one of the classes in $\calA_1$.

We similarly can define elements $c_i$, $i>1$ in the following way: Let
$D=\bigcap_{j<i} C_A(b_j^{hc_1\cdots c_{i-1}})$ and let $\{y_j\}$ be a set of
representatives of the $D$ classes of elements of $A$, then define $c_i$
such that $(b_j^{hc_1\cdots c_{i-1}})^c_i=y_j$ is one of the chosen
representatives. (Clearly we must have that $y_j$ will lie in one of the
classes in $\calA_i$.)

This shows (choosing the $y_j$ in step~\ref{schritt3} as representatives)
that for all $i$ there exists $c_i\in A$ such that
\[
(b_i^h)^{c_1c_2\cdots c_k}=(b_i^h)^{c_i\cdots c_k}=b_i^g
\]
for one of the conjugating elements $g$ given in $\calB$. In particular
$(B^h)^{c_1c_2\cdots c_k}=B^g$ with $c_1c_2\cdots c_k\in A$, proving the
claim.
\end{proof}

To find {\em} all conjugates of $B$ within $A$ we finally form the
$A$-orbits of the subgroups in $\calB$. Note that we do not guarantee that
the subgroups on $\calB$ are not conjugate (or even different). In such a
case there must be elements of $A$ that normalize an embedded
subgroup $B^g\le A$ and induce automorphisms. The potential for this
happening can be checked for a priori by $\gpind{N_G(B)}{C_G(B)}$.
\bigskip

In studying the algorithm, the reader will notice that the inclusion $B^g\le
A$ is simply a consequence of the fact that $\bigcup \calA_i\subset A$. We
thus could replace $A$ with a larger subgroup $L$, normalizing $A$ such as
$L=N_G(A)$, by  initializing the classes $\calA_i$ with the $L$-classes within
$A$ that intersect with $b_i^G$ and assigning $D=C\cap L$ in
step~\ref{schritt2}.

A call to ${\tt Search}(G,1,1_G)$ the produces a list of subgroups, such
that every conjugate $B^g\le A$ is $L$-conjugate to one of the subgroups in
the list. Conjugacy tests then can be used to obtain
representatives of the $L$-classes.

\subsection{Embedding conjugates}

We now consider the dual problem of determining conjugates $A^g$ of a
subgroup $A\le G$ that contain a given subgroup. This will be implemented by
the routine ${\tt EmbeddingConjugates}(G,A,B)$ that has been referred to
already above.
\smallskip

We have that $A^{g^{-1}}\ge B$ if and only if $B^{g}\le A$.
This duality can be translated to double cosets, where 
the double cosets $\dnk{N_G(A)}{G}{N_G(B)}$ are given as sets of inverses of
the double cosets $\dnk{N_G(B)}{G}{N_G(A)}$. Thus inverting representatives
for one set of double cosets yields representatives for the other.

The double cosets $\dnk{N_G(A)}{G}{N_G(B)}$ correspond to $N_G(B)$-orbits on
the $G$-conjugates of $A$; the double cosets $\dnk{N_G(B)}{G}{N_G(A)}$ to
$N_G(A)$ orbits on conjugates of $B$. Together this shows:

\begin{lemma}
Let $\{g_i\}_i$ be a set of elements such that the set groups $B^{g_i}$
contains representatives of the $N_G(A)$-classes of conjugates $B^g\le A$.
Then the set of subgroups $A^{g_i^{-1}}$ contains representatives of the
$N_G(B)$-orbits of conjugates $A^g>B$.
\end{lemma}

A set of such elements $g_i$ satisfying the former condition was obtained in
the previous section. This gives the following algorithm 
${\tt EmbeddingConjugates}(G,A,B)$:

\begin{enumerate}[1.]
\item Let $L=N_G(A)$.
\item Select a generating set $\{b_1,\ldots,b_k\}$ of $B$. For each $i$
let $\calA_i$ be a list of the $L$-conjugacy classes that partition 
$A\cap b_i^G$.
\item Set $\calB=[\,]$ and call ${\tt Search}(G,1,1_G)$.
\item Let $\calC=\{A^{g^{-1}}\mid (B^g,g)\in\calB\}$.
\item Return the union of the $N_G(B)$ orbits of the elements of $\calC$.
\end{enumerate}

It would be easy to also keep track of conjugating elements for the
subgroups in $\calC$.

\section{Examples}
\label{xampl}

The algorithm, as described, has been implemented by the author in {\sf
GAP}~\cite{GAP4} and will be available as part of the 4.9 release (through
significantly improved performance of the operation {\tt
IntermediateSubgroups}).
\medskip

We indicate the performance of the algorithm
in a number of examples, in particular in comparison to
the old, block-based, method discussed at the end of the first section.
The examples have been chosen primarily for being easily reproducible (without a
need to list explicit generators) cases of
non-maximal subgroups with a moderate number of intermediate subgroups.

While the algorithm per se does not make assumptions about the way the group
is represented, most of the examples were chosen as permutation groups as for these
groups the practically usable implementations in {\sf GAP} of the
underlying routines,  in particular element conjugacy and double cosets,
perform more smoothly than for matrix groups.

For polycyclic groups the
number of intermediate subgroups in general larger.

In the following examples,
the notation $\mbox{Syl}_p$
indicates a $p$-Sylow subgroup.


If a subgroup is given by a structure this implies that there is a unique
such subgroup up to automorphisms. If there are two such subgroups the
different cases will be distinguished as \#1 and \#2.

{\#} counts the number of proper intermediate subgroups (i.e.
excluding $G$ and $U$).
$t$ is the runtime in seconds on a 3.7 GHz 2013 MacPro with ample memory.

When maximal subgroups of simple groups were required, all examples
utilized a lookup, such a calculation thus did not contribute significantly
to the overall runtime. 

Table~\ref{smallxam} shows comparative timings to the old (block-based)
algorithm. The timings indicate that for indices larger than a few hundred
the new method is universally superior to the old one, while for smaller
indices the naive block-based approach works faster.
We have not examined examples with indices smaller than hundred, as these
often involve maximal subgroups or factor groups of significantly smaller
order.

The solvable groups given in the last lines are maximal
subgroups of $\mbox{Fi}_{22}$, respectively $\mbox{Co}_1$,
given by a pc presentation. 

\medskip

\begin{table}
\begin{center}
\begin{tabular}{l|l|r|r|r|r}
$G$
&$U$
&Index
&\#
&$t_{\mbox{Old}}$
&$t_{\mbox{New}}$
\\
\hline
$S_6$
&$1$
&$2^{4}3^{2}5$
&$1453$
&2.6
&2.8
\\
$A_7$
&$2$
&$2^{2}3^{2}5
{\cdot}7$
&$156$
&52
&0.8
\\
$S_5\wr S_2$
&$\mbox{Syl}_5$
&$2^{7}3^{2}$
&$58$
&5.7
&2
\\
$HS$
&$S_7$
&$2^{5}5^{2}11$
&$3$
&1.4
&1.9
\\
$HS$
&$\mbox{Solv}1152$
&$2^{2}5^{3}7
{\cdot}11$
&$2$
&6
&2.2
\\
$\mbox{PSL}_4(3)$
&$A_6 (\#1)$
&$2^{4}3^{4}13$
&$1$
&2.3
&1
\\
$\mbox{PSL}_4(3)$
&$A_6 (\#2)$
&$2^{4}3^{4}13$
&$17$
&28
&4
\\
$\mbox{Sp}_6(2)$
&$\mbox{PSL}_3(2) (\#1)$
&$2^{6}3^{3}5$
&$6$
&2
&1
\\
$\mbox{Sp}_6(2)$
&$\mbox{PSL}_3(2) (\#2)$
&$2^{6}3^{3}5$
&$9$
&4.5
&3.1
\\
$3^{1+6}{:}2^{3+4}{:}3^2{:}2$
&$\mbox{Syl}_3$
&$2^8$&$19$
&0.2&1.5\\
$3^{1+6}{:}2^{3+4}{:}3^2{:}2$
&$\mbox{Syl}_2$
&$3^9$&$37$
&231&191\\
$3^{3+4}{:}2(S_4{\times}S_4)$
&$\mbox{Syl}_2$
&$3^9$&$15$
&254&1.3\\
$3^{3+4}{:}2(S_4{\times}S_4)$
&$\mbox{Syl}_3$
&$2^7$&$30$
&0.2&5\\

\end{tabular}
\end{center}
\caption{Comparison between old and new method}
\label{smallxam}
\end{table}

Table~\ref{xamples} gives examples for some cases of significantly larger
index for which the old algorithm would not terminate within reasonable
time, thus only timings for the new method are given.

\begin{table}

\begin{center}
\begin{tabular}{l|l|l|r|r|r}
$G$
&Degree
&$U$
&Index
&\#
&$t$
\\
\hline
$S_{11}$
&$11$
&$11:{5}$
&$2^{7}3^{4}5^{2}7
{\cdot}11$
&$4$
&0.5
\\
$S_{10}\times S_{10}$
&$20$
&$A_{10}\mbox{(diag)}$
&$2^{9}3^{4}5^{2}7$
&$5$
&8
\\
$HS$
&$100$
&$\mbox{Syl}_7$
&$2^{9}3^{2}5^{3}11$
&$41$
&9
\\
$HS$
&$100$
&$\mbox{Syl}_3$
&$2^{9}5^{3}7
{\cdot}11$
&$249$
&12
\\
$HS$
&$100$
&$(\mbox{Syl}_2)'$
&$2^{3}3^{2}5^{3}7
{\cdot}11$
&$57$
&6
\\
$S_{24}$
&$24$
&$24:8$
&$2^{16}3^{9}5^{4}7^{3}11^{2}13
{\cdot}17
{\cdot}19
{\cdot}23$
&$409$
&297
\\
$S_{25}$
&$25$
&$C_5\times C_5$
&$2^{22}3^{10}5^{4}7^{3}11^{2}13
{\cdot}17
{\cdot}19
{\cdot}23$
&$2734$
&418
\\
$S_{25}$
&$25$
&$C_{25}$
&$2^{22}3^{10}5^{4}7^{3}11^{2}13
{\cdot}17
{\cdot}19
{\cdot}23$
&$127$
&129
\\
$Co_3$
&$276$
&$\mbox{Syl}_2$
&$3^{7}5^{3}7
{\cdot}11
{\cdot}23$
&$12$
&316
\\
$Co_3$
&$276$
&$\mbox{Syl}_7$
&$2^{10}3^{7}5^{3}11
{\cdot}23$
&$396$
&310
\\
$L_5(7)$
&$2801$
&$\mbox{Syl}_{2801}$
&$2^{11}3^{5}5^{2}7^{10}19$
&$1$
&168
\\
$L_8(2)$
&$255$
&$\mbox{Syl}_{127}$
&$2^{28}3^{5}5^{2}7^{2}17
{\cdot}31$
&$8$
&100
\\
$L_8(2)$
&$255$
&$\mbox{Syl}_7$
&$2^{28}3^{5}5^{2}17
{\cdot}31
{\cdot}127$
&$14074$
&3305
\\
\end{tabular}

\end{center}
\caption{Larger example calculations}
\label{xamples}
\end{table}

As expected, the timings show that the cost is dominated by the number of subgroups to be
found rather than permutation degree.
\medskip

As an illustration of use, consider the case of a $2$-Sylow subgroup of
$Co_3$. The {\sf GAP} calculation:
\begin{verbatim}
gap> g:=SimpleGroup("Co3");;s:=SylowSubgroup(g,2);
<permutation group of size 1024 with 10 generators>
gap> int:=IntermediateSubgroups(g,s);
rec( inclusions := [[0, 1], [0, 2], [0, 3], [0, 4], [1, 7],
  [1, 8], [1, 10], [2, 5], [2, 10], [3, 5], [3, 6], [3, 7],
  [4, 6], [4, 8], [5, 9], [5, 12], [6, 9], [6, 11], [7, 11],
  [7, 12], [8, 11], [9, 13], [10, 12], [11, 13], [12, 13] ],
subgroups := [<permutation group of size 3072 with 9 generators>,
    [ ... a list of 12 subgroups ]
gap> List(int.subgroups,Size);
[ 3072,3072,3072,3072,9216,9216,21504,21504,27648,46080,
  322560,2903040 ]
\end{verbatim}
returns the subgroups lying between $s$ and $g$, numbered $1$ to $12$, as
well as a list of maximal inclusion relations that is depicted in
figure~\ref{conwayincl}.
\begin{figure}
\begin{center}
\includegraphics[width=5cm]{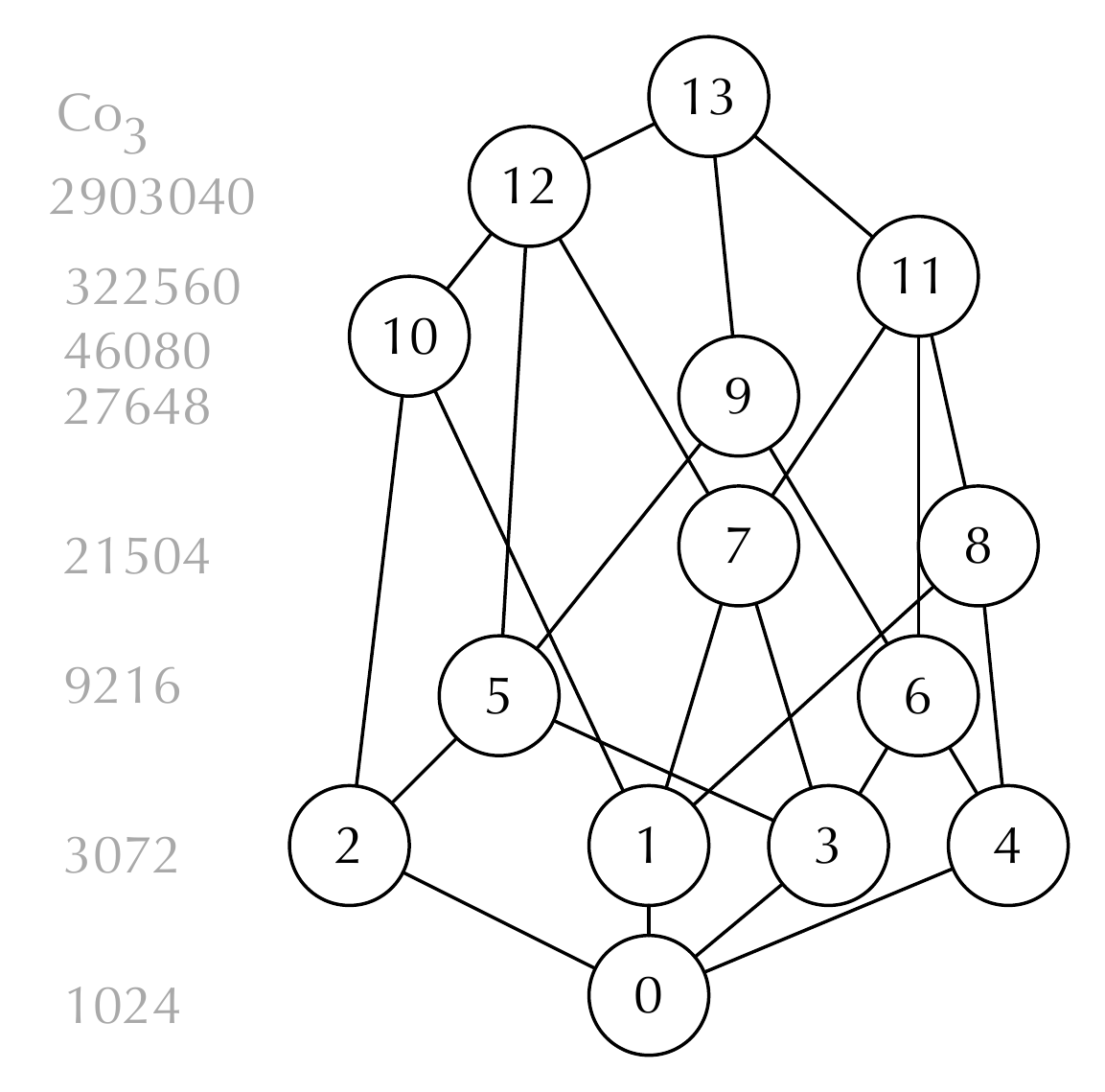}
\end{center}
\caption{Intermediate subgroups between $Co_3$ and a $2$-Sylow subgroup.}
\label{conwayincl}
\end{figure}

\section*{Acknowledgement}
The author's work has been supported in part by
Simons' Foundation Collaboration Grant~244502.


\begin{thebibliography}{CHSS05}

\bibitem[Asc08]{aschbacher08}
Michael Aschbacher.
\newblock On intervals in subgroup lattices of finite groups.
\newblock {\em J. Amer. Math. Soc.}, 21(3):809--830, 2008.

\bibitem[BS92]{bealsseress92}
Robert Beals and {\'A}kos Seress.
\newblock Structure forest and composition factors for small base groups in
  nearly linear time.
\newblock In {\em {Proceedings of the {$24^{\rm th}$} ACM Symposium on Theory
  of Computing}}, pages 116--125. ACM Press, 1992.

\bibitem[CCH01]{coxcannonholtlatt}
John Cannon, Bruce Cox, and Derek Holt.
\newblock Computing the subgroup lattice of a permutation group.
\newblock {\em J. Symbolic Comput.}, 31(1/2):149--161, 2001.

\bibitem[CH04]{holtcannonmaxgroup}
John Cannon and Derek Holt.
\newblock Computing maximal subgroups of finite groups.
\newblock {\em J. Symbolic Comput.}, 37(5):589--609, 2004.

\bibitem[CHSS05]{cannonholtslatterysteel}
John~J. Cannon, Derek~F. Holt, Michael Slattery, and Allan~K. Steel.
\newblock Computing subgroups of bounded index in a finite group.
\newblock {\em J. Symbolic Comput.}, 40(2):1013--1022, 2005.

\bibitem[DeM12]{deMeoThesis}
Willian~J. DeMeo.
\newblock {\em Congruence Lattices of Finite Algebras}.
\newblock PhD thesis, The University of Hawai'i at M\={a}noa, 2012.

\bibitem[DFH17]{detinkoflanneryhulpke17}
A.~S. Detinko, D.~L. Flannery, and A.~Hulpke.
\newblock Zariski density and computing in arithmetic groups.
\newblock {\em Math. Comp.}, accepted 2017.
\newblock https://doi.org/10.1090/mcom/3236.

\bibitem[DM88]{dixonmajeed}
John~D. Dixon and Abdul Majeed.
\newblock Coset representatives for permutation groups.
\newblock {\em Portugal. Math.}, 45(1):61--68, 1988.

\bibitem[EH01]{eickhulpke01}
Bettina Eick and Alexander Hulpke.
\newblock Computing the maximal subgroups of a permutation group {I}.
\newblock In William~M. Kantor and {\'A}kos Seress, editors, {\em Proceedings
  of the International Conference at {The} {Ohio} {State} {University}, {June}
  15--19, 1999}, volume~8 of {\em Ohio State University Mathematical Research
  Institute Publications}, pages 155--168, Berlin, 2001. de Gruyter.

\bibitem[GAP16]{GAP4}
The GAP~Group, \url{http://www.gap-system.org}.
\newblock {\em {GAP -- Groups, Algorithms, and Programming, Version 4.8.6}},
  2016.

\bibitem[GS63]{graetzerschmidt63}
G.~Gr\"atzer and E.~T. Schmidt.
\newblock Characterizations of congruence lattices of abstract algebras.
\newblock {\em Acta Sci. Math. (Szeged)}, 24:34--59, 1963.

\bibitem[HEO05]{holtbook}
Derek~F. Holt, Bettina Eick, and Eamonn~A. O'Brien.
\newblock {\em Handbook of {Computational Group Theory}}.
\newblock Discrete Mathematics and its Applications. Chapman \& Hall/CRC, Boca
  Raton, FL, 2005.

\bibitem[Hul96]{hulpkediss}
Alexander Hulpke.
\newblock {\em Konstruktion transitiver {P}ermutationsgruppen}.
\newblock PhD thesis, Rheinisch-Westf{\accent127 a}lische Tech\-ni\-sche
  Hoch\-schule, Aachen, Germany, 1996.

\bibitem[Hul99]{hulpkeinvar}
Alexander Hulpke.
\newblock Computing subgroups invariant under a set of automorphisms.
\newblock {\em J. Symbolic Comput.}, 27(4):415--427, 1999.
\newblock (ID jsco.1998.0260).

\bibitem[Hul13]{hulpketfsub}
Alexander Hulpke.
\newblock Calculation of the subgroups of a trivial-fitting group.
\newblock In {\em I{SSAC} 2013---{P}roceedings of the 38th {I}nternational
  {S}ymposium on {S}ymbolic and {A}lgebraic {C}omputation}, pages 205--210.
  ACM, New York, 2013.

\bibitem[Lau82]{laue82}
Reinhard Laue.
\newblock Computing double coset representatives for the generation of solvable
  groups.
\newblock In Jacques Calmet, editor, {\em EUROCAM '82}, volume 144 of {\em
  Lecture Notes in Computer Science}. Springer, 1982.

\bibitem[Leo97]{leon97}
Jeffrey~S. Leon.
\newblock Partitions, refinements, and permutation group computation.
\newblock In Larry Finkelstein and William~M. Kantor, editors, {\em Proceedings
  of the 2nd {DIMACS} Workshop held at {Rutgers University}, {New Brunswick,
  NJ}, {June} 7--10, 1995}, volume~28 of {\em {DIMACS}: Series in Discrete
  Mathematics and Theoretical Computer Science}, pages 123--158. American
  Mathematical Society, Providence, RI, 1997.

\bibitem[LR11]{longreid11}
D.~D. Long and A.~W. Reid.
\newblock Small subgroups of {${\rm SL}(3,\mathbb{Z})$}.
\newblock {\em Exp. Math.}, 20(4):412--425, 2011.

\bibitem[Neu60]{neubueser60}
Joachim Neub{\"u}ser.
\newblock {Untersuchungen des Untergruppenverbandes endlicher Gruppen auf einer
  {pro\-gramm\-ge\-steu\-er\-ten} elektronischen Dualmaschine}.
\newblock {\em Numer. Math.}, 2:280--292, 1960.

\bibitem[P{\'al}f95]{palfy95}
P.~P. P{\'al}fy.
\newblock Intervals in subgroup lattices of finite groups.
\newblock In C.~M. Campbell, T.~C. Hurley, E.~F. Robertson, S.~J. Tobin, and
  J.~J. Ward, editors, {\em Groups '93 Galway/St~Andrews}, volume 212 of {\em
  London Mathematical Society Lecture Note Series}, pages 482--494. Cambridge
  University Press, 1995.

\bibitem[PP80]{palfypudlak}
P\'eter~P\'al P\'alfy and Pavel Pudl\'ak.
\newblock Congruence lattices of finite algebras and intervals in subgroup
  lattices of finite groups.
\newblock {\em Algebra Universalis}, 11(1):22--27, 1980.

\bibitem[{Sch}90]{schmalz90}
Bernd {Schmalz}.
\newblock {Verwendung von Untergruppenleitern zur Bestimmung von
  Doppelnebenklassen.}
\newblock {\em Bayreuth. Math. Schr.}, 31:109--143, 1990.

\bibitem[Sha03]{shareshian03}
John Shareshian.
\newblock Topology of order complexes of intervals in subgroup lattices.
\newblock {\em J. Algebra}, 268(2):677--686, 2003.

\bibitem[The97]{Theissen97}
Heiko Thei{\ss}en.
\newblock {\em {Eine Methode zur Nor\-ma\-li\-sa\-tor\-be\-rech\-nung in
  Per\-mu\-ta\-tions\-gruppen mit Anwendungen in der Konstruktion primitiver
  Gruppen}}.
\newblock Dissertation, Rheinisch-Westf{\accent127 a}lische Tech\-ni\-sche
  Hoch\-schule, Aachen, Germany, 1997.

\bibitem[Wat96]{watatani96}
Yasuo Watatani.
\newblock Lattices of intermediate subfactors.
\newblock {\em J. Funct. Anal.}, 140(2):312--334, 1996.

\end{thebibliography}

\end{document}